\documentclass[12pt]{amsart}
\usepackage{amsmath}
\usepackage{amssymb}
\usepackage{graphicx}
\usepackage[colorlinks=true, allcolors=blue]{hyperref}
\usepackage[T1]{fontenc}
\usepackage{tikz-cd}
\usepackage{pifont}
\usepackage{calligra}
\usepackage{calrsfs}
\DeclareMathAlphabet{\pazocal}{OMS}{zplm}{m}{n}
\usepackage{xcolor}
\usepackage{appendix}
\usepackage[margin=1.2in]{geometry}%
\usepackage{hyperref}
\usepackage{mathrsfs} 

\newtheorem{dfn}{Definition}
\numberwithin{dfn}{subsubsection}
\newtheorem{Theorem}{Theorem}
\numberwithin{Theorem}{subsubsection}
\newtheorem*{Theorem-non}{Theorem}
\newtheorem{rem}{Remark}
\numberwithin{rem}{subsubsection}
\newtheorem{lem}{Lemma}
\numberwithin{lem}{subsubsection}
\newtheorem{cor}{Corollary}
\numberwithin{cor}{subsubsection}
\newtheorem{prop}{Proposition}
\numberwithin{prop}{subsubsection}
\numberwithin{equation}{subsubsection}

\newcommand{\bA}{{\bf A}}
\newcommand{\bK}{{\bf K}}

\newcommand{\cO}{\mathcal{O}}

\newcommand{\K}{\mathbb{K}}

\newcommand{\A}{\mathcal{A}}
\newcommand{\C}{\mathbb{C}}
\newcommand{\hH}{\mathbb{H}}
\newcommand{\R}{\mathbb{R}}
\newcommand{\Z}{\mathbb{Z}}

\newcommand{\cT}{\mathcal{T}}

\newcommand{\sA}{\mathscr{A}}
\newcommand{\sE}{\mathscr{E}}

\newcommand{\cA}{\tilde{\mathscr{A}}}
\newcommand{\con}[1]{{\stackrel{\scriptscriptstyle{#1}}{\nabla}}\phantom{}} 

\title{On the geometry of K\"ahler--Frobenius manifolds $\&$
their classification}

\author{Noemie C. Combe}
\address{ University of Warsaw, Department of mathematics, MIMUW}
\email{n.combe@uw.edu.pl}
\thanks{The author acknowledges support from the project No. 2022/47/P/ST1/01177 co-founded by the National Science Centre and the European Union’s Horizon 2020 research and innovation program, under the Marie Sklodowska Curie grant agreement
No. 945339.\includegraphics[scale=0.15]{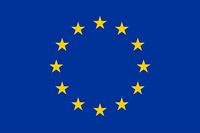}. The author thanks A. Szczepański and M.Halenda for interesting discussions about flat K\"ahler manifolds and for an invitation to the university of Gdańsk.}

\begin{document}

\begin{abstract}
The purpose of this article is to show that flat compact K\"ahler manifolds exhibit the structure of a Frobenius manifold, a structure originating in 2D Topological Quantum Field Theory and closely related to Joyce structure. As a result, we classify all such manifolds. It can be deduced that K\"ahler--Frobenius manifolds include certain Calabi--Yau manifolds, complex tori $T=\C^n/\Z^n$, generalized Hantzsche--Wendt manifolds, hyperelliptic manifolds and manifolds of type $T/G$, where $G$ is a finite group acting on $T$ freely and containing no translations. An explicit study is provided for the two-dimensional case. Additionally, we can prove that Chern's conjecture for K\"ahler pre-Frobenius manifolds holds. Lastly, we establish that certain classes of K\"ahler-Frobenius manifolds share a direct relationship with theta functions which are important objects in number theory as well as complex analysis. 
\end{abstract}
\maketitle
\,


\, 


\,

\section{Introduction}

\subsection{}
The motivation for this article inscribes itself in the interaction between (differential) geometry and  Quantum Field theories. A part of this research program, initiated by E. Witten \cite{Witten1,Witten2} and Yu. I. Manin \cite{Man1,Man2}, led to fruitful interactions between quantum cohomology, Frobenius manifolds and Verlinde algebras which count theta functions. 

\, 

As mentioned in \cite[p.45]{Witten1}: {\sl ``If one considers a K\"ahler manifold $X$ with $c_1(X)=0$, then the supersymmetric sigma model is conformally invariant. Such models provide classical solutions to string theory."} Therefore, it is natural to consider K\"ahler manifolds with vanishing first Chern class $c_1(X)=0$.

\, 

We demonstrate in this article that flat compact K\"ahler manifolds (i.e. with $c_1=0$) possess the structure of a K\"ahler--Frobenius manifold. 
This involve isomonodromic pencils of flat connections on the tangent bundle of a complex manifold $M$. The Frobenius structure is also referred often as the {\it associativity equation} due to the fact that one has the structure of an associative, commutative algebra of Frobenius type on the tangent sheaf.  A close notion to the Frobenius structure is the Joyce structure \cite{J07}.

\, 

Classes of K\"ahler--Frobenius manifolds come naturally equipped with theta functions. This gives an explicit bridge between the notion of Frobenius manifolds and theta functions. 

\, 

A first example of such manifolds includes abelian varieties/complex tori. This is particularly interesting in relation to theta functions, since given a complex torus $T$  of dimension $g$, a line bundle $L$ defining a principal polarization and  a positive integer $s$, the space of level $s$ theta functions is given by $H^0(T,L^{\otimes s})$. 

\, 

The explicit relation between classes of Frobenius manifolds (being flat K\"ahler manifolds with $c_1=0$) and theta functions follows from Theorem \ref{T:Class} and Theorem \ref{T:Main}.

\subsection{} Flat K\"ahler manifolds are naturally related to models providing classical solutions to string theory. Therefore, it is natural to check whether they satisfy certain highly non-linear partially differential equations, the so-called  Witten--Dijkgraaf--Verlinde--Verlinde equations (WDVV equations). This equation were discovered in the 1990s in relation with the classification of 2 dimensional Topological Field Theories \cite{DVV,Du,Man1,Man2}. These equations are defined for Riemannian manifolds \cite[Lecture 1]{Du}. Much attention has been given to them due to their central role in the mathematical mirror symmetry problem.

\,

The notion of a Frobenius manifold is a {\it geometrization} of the WDVV Partial Differential equation (see \cite[Sec. 1, p.19]{Man1}). This terminology distinguishes the geometrical viewpoint from the analytical one (PDE equations). Working with {\it Hermitian manifolds} leads to requiring {\it extra} properties. In particular, we introduce a {\it hermitian WDVV equation}. 

\, 

By Lem. \ref{L:Kahler}, a manifold satisfying such an equation is {\it necessarily} K\"ahler and equipped with a holomorphic affine structure. The hermitian WDVV equation is as follows:
\begin{equation}\label{E:HWDVV}
 \forall a,b,\bar{c},\bar{d}, \quad \sum_{\bar{e}f}\Phi_{ab\bar{e}}g^{\bar{e}f}\Phi_{f\bar{c}\bar{d}}=\sum_{\bar{e}f} \Phi_{b\bar{c}\bar{e}}g^{\bar{e}f}\Phi_{fa\bar{d}},
 \end{equation}
where $\Phi$ is a real potential function and $ \Phi_{ij\bar{k}}=\frac{\partial^3\Phi}{\partial{z}^{i}\partial{z}^{j}\partial \bar{z}^k}$; $g$ is a non-degenerate symmetric bilinear form which is a K\"ahler metric. 

We shall discuss these manifolds in Sec.~\ref{S:WDVV} where we introduce a new object called {\it Frobenius bundles}. This new object provides a more practical formalism for describing the structure of Frobenius manifolds.

    \subsection{Results} In Theorem \ref{T:Main} we prove that {\it if a compact complex K\"ahler manifold $M$  has vanishing curvature then it is a Frobenius manifold}. K\"ahler Frobenius manifolds imply $c_1(M)=0$. But the converse is not true.  Examples of such K\"ahler--Frobenius manifolds include K\"ahler--Einstein manifolds (see Sec. \ref{S:3.1}, Proposition \ref{P:Kahler--Einstein}) and certain Calabi--Yau manifolds.
   
    \, 
    
    A classification of all K\"ahler- Frobenius manifolds is available in Theorem \ref{T:Class}. The list of complex Frobenius manifolds covers the following geometric objects: 

\begin{enumerate}
    \item complex tori $T=\C^n/\Z^n$, where $n\geq 1$; 
\item manifolds of type $T/G$, where $G$ is a finite group acting on $T$ freely and containing no translations. 
\item Generalized Hantzsche--Wendt manifolds (orientable); 
\item a certain family of Calabi--Yau manifolds 
\item  hyperelliptic manifolds.

\end{enumerate}
Those classes are particularly rich from the viewpoint of analysis and number theory. On the one hand side relations to Frobenius manifolds allow those object to satisfy a Monge--Ampère type of equation \cite{C24}. This can be  easily related to 
the Laplace eigenvalue problem for a domain $\mathscr{D}\subset \R^{2n}$ with the Dirichlet boundary
condition. In this setting, the problem is to find all functions $u:\overline{\mathscr{D}} \to \C $ not identically zero and satisfying 
$$\Delta u(x)=K x, \forall x\in \mathscr{D}$$ and some constant $K$ ($u$ vanishes on the boundary of the domain). For such Frobenius manifolds $T/G$ the eigenfunctions turn out to be trigonometric functions \cite{Cryst}. Note that these topics can be further related to the Fuglede and Goldbach conjectures.

On the other hand the classes related to the torus have direct relations to the multiplicative group of theta functions. If $V$ is a complex vector space  (dimension $n$) and $\Lambda$ is a lattice in $V$, then the quotient $T = V/\Lambda$ forms a complex torus. 
Take $L : V \times \Lambda \to \C$ and $J : \Lambda \to \C$ to be maps such that $L(x,l)$ are linear in $x$ for all element $l\in \Lambda$. A theta function on $V$ with respect to $\Lambda$ of type $(L,J)$ is a meromorphic function $H: V \to \C$ satisfying the relation
\[H(x+l)=e(L(x,l)+J(l))H(x)\]for all $l\in \Lambda$,
where $e(x)=e^{2\pi\imath x}$. The theta functions form a multiplicative group. Theta functions of same type form a vector space over $\C$. 

\, 

\subsection{Plan with results}

Section \ref{S:WDVV} introduces Frobenius bundles for real/complex numbers and Frobenius algebra bundles, as well as all the necessary material for the unfolding of our results ( \ref{S:Fbund}--\ref{S:1.3}).  The advantage of introducing Frobenius bundles is that it potentially allows a generalisation of the notion of Frobenius manifolds to manifolds over other algebras, of finite dimension. We present the key statement (Theorem \ref{T:equ}). The proof is outlined in the next section.

\,

 In section~\ref{S2:Proof}, we prove the statement mentioned in the first section i.e. Theorem \ref{T:equ}.  Sections \ref{S:HermitianFman} provides the complex version of the previous subsection \ref{S2:realcase}.

\,

In section~\ref{S:KFmfd}, we consider K\"ahler-Frobenius manifolds. We prove the Theorem \ref{T:Main} which states that a complex compact K\"ahler manifold with vanishing curvature is a  Hermitian Frobenius manifold and we discuss some properties around this statement. In particular, by  K\"ahler-Frobenius manifold we mean a K\"ahler manifold being a Hermitian Frobenius manifold. In Sec. \ref{S:Chern} we prove that Chern's conjecture holds for pre-Frobenius K\"ahler manifolds (Theorem \ref{T:Chern}).

\,

In section~\ref{S:Properties}, we investigate properties of K\"ahler-Frobenius manifolds, we particularly study the K\"ahler--Einstein case (section~\ref{S:3.1}). We consider some properties in Section \ref{S:Props} related to coverings by a torus and pre-Frobeniusity.

\, 

In section~\ref{S:2D}, the study of 2 dimensional  surfaces is provided, see Theorem \ref{T:ClassSurf}. We study all complex surfaces and discuss whether they can qualify as a Frobenius surface or not. The cases of complex tori and hyperelliptic curves satisfy the required properties, whereas K3 surfaces, ruled surfaces, Hopf surfaces do not. 

\, 

In Section~\ref{S:ClassN} we classify the K\"ahler-Frobenius manifolds. This leads to the classification theorem \ref{T:Class}. We study the pencil of connections for K\"ahler-Frobenius and show that they are of Hermitian--Einstein type. The final subsection concludes the paper with some final observations about Joyce structures and theta functions. 

\smallskip

\section{Frobenius bundles}\label{S:WDVV}
 A Frobenius manifold $M$ is a {\it geometrization} of the WDVV Partial Differential  Equation. The former is a geometric object; the latter is an analytical one. Both viewpoints coincide \cite{Man1,Man2}. 

\, 

According to \cite{Dub96} the  manifold $M$ admits the structure of a Frobenius manifold if:
\begin{itemize}
\item at any point of $M$ the tangent space has the structure of a Frobenius algebra; 
\item the invariant inner product $\langle-,-\rangle $ is a flat metric on $M$;
\item the unity vector field satisfied $\con{0}e=0$;
\item The rank 4 tensor $(\con{0}_W A) (X,Y,Z)$ is fully symmetric;
\item The vector field $E$ is determined on M such that $\con{0}(\con{0}E)=0$. 

\end{itemize}

The Euler vector field belongs to the class of affine vector fields. Its existence is tightly related to the fact that we have initially an affine structure on $M$. To see this it is enough to recall that the following statements are equivalent.

Let $\con{0}$ be the corresponding affine flat torsionless connection.
\begin{enumerate}
 \item $E$ is an {\it affine} vector field.
 \item $\con{0}(\con{0} E)=0$.
    \item $\con{0}_Y\con{0}_Z E=\con{0}_{\con{0}_YZ}E$ for all vector fields $Y,Z$ on $M$.
    \item The coefficients of $E$ are affine functions. Given, $$E=\sum_m E^{m}\partial_m$$ we have $E^{m}=a^m_j x^j + b^m$, where $a^m_j$ and $b^m$ are constants in $\R$.
\end{enumerate}

\, 

We shall therefore adopt a more concise definition of Frobenius manifolds, via Frobenius bundles, that stems from the definition of \cite[p.19]{Man1}.

The latter approach is constructed by first introducing the notion of a potential pre-Frobenius  manifold
(that is manifolds satisfying the pre-Frobenius axioms developed by Yu. Manin in \cite[p.19]{Man1} as well as the potentiality axiom) and then adding an extra requirement that the algebraic structure on the tangent sheaf $(T_M,\circ)$ is everywhere a Frobenius algebra (commutative, associative, unital with a bilinear symmetric map satisfying 
$\langle x\circ y, z\rangle = \langle x,y \circ z\rangle$ for $x,y,z$ elements in the algebra). In other words, we have an affine space of flat connections on $M$.

\, 

The topic of Frobenius manifolds has stimulated very different approaches and results \cite{CoMa,CoVa,Dub96,Fei,Hit,KaKP,KoMa,Man05,Man98,KonoMa,Sa,To04} (to cite only a few examples). 
 
\subsection{Frobenius bundles}\label{S:Fbund}

 \, 
 We introduce the notion of Frobenius bundles, as an attempt towards a more practical and geometric approach. 
 
 We argue that this tool allows a more general definition, for example if we want to extend the notion of Frobenius manifolds for manifolds defined over an algebra of finite dimension and will be the subject of future investigations. 

 \smallskip 
 
This Section is devoted to the introduction of central objects in the following statement.  
\begin{Theorem}\label{T:equ}
Let $(M,g)$ be a K\"ahler manifold. Then, the following statements are equivalent:
\begin{enumerate}

\item the Hermitian WDVV equation is satisfied;
    \item $M$ is a Hermitian Frobenius manifold;
    \item $M$ is equipped with a Frobenius bundle.
\end{enumerate} 
\end{Theorem}
To prove this statement, we start by introducing the Hermitian WDVV equation (Sec. \ref{S:WDVV}). Prerequisites for Frobenius manifolds are outlined in Sec. \ref{S:1.2}. Frobenius bundles are defined in Sec.\ref{S:1.3}. The extra data provided by the complex structure is explained in Sec.\ref{S:HermitianFman}.

We show that (1) is equivalent to (2) in Lem. \ref{L:1to2}. The equivalence between (2) and (3) is proved separately in the real and complex case. In the real case, this is Lem. \ref{L:2-3R}. In  the complex, it is Lem. \ref{L:2-3C}. 

\medskip 
\subsection{Recollections}
We refer to \cite{YK} for recollections on the realm of complex manifolds  (see Fig. \ref{fig:relations}). Let $M$ be a differentiable manifold. A tensor field $J$ is an almost complex structure if for every point $x\in M$, $J$ is an endomorphism of the tangent space at $x$ such that $J^2=-Id$. Every almost complex manifold has even (real) dimension. 

\smallskip 

Let $(M,J)$ be an almost complex manifold. If $M$ is equipped with a Riemannian  metric $g$ such that $g(JX,JY)=g(X,Y)$ for any vector fields $X, Y$ on $M$ then it is called an almost Hermitian manifold. A complex manifold with a Hermitian metric is a Hermitian manifold.

\smallskip

Let $M$ be as above. It is a complex manifold if and only if $M$ admits a linear connection $\nabla$ such that $\nabla J=0$ and the torsion is null.  

\smallskip 

Let $(M,J)$ be an almost complex manifold with Hermitian metric $g$. Let $\nabla$ be the covariant differentiation of the Riemannian connection defined by $g$.
Let the fundamental 2-form $\varphi$ on M be defined as $\varphi(X,Y)=g(X,JY)$ for all vector fields $X,Y$ on $M$.
The Hermitian metric is called a K\"ahler metric if the fundamental 2-form $\varphi$ is closed.  An almost complex manifold with a K\"ahler metric is an almost K\"ahler manifold.  A complex manifold with a Kahler metric is a K\"ahler manifold. 
 
 \smallskip

The almost complex structure becomes pseudo-complex if the Nijhenuis tensor vanishes. 
\begin{figure}[ht]
    \centering
\includegraphics[scale=0.6]{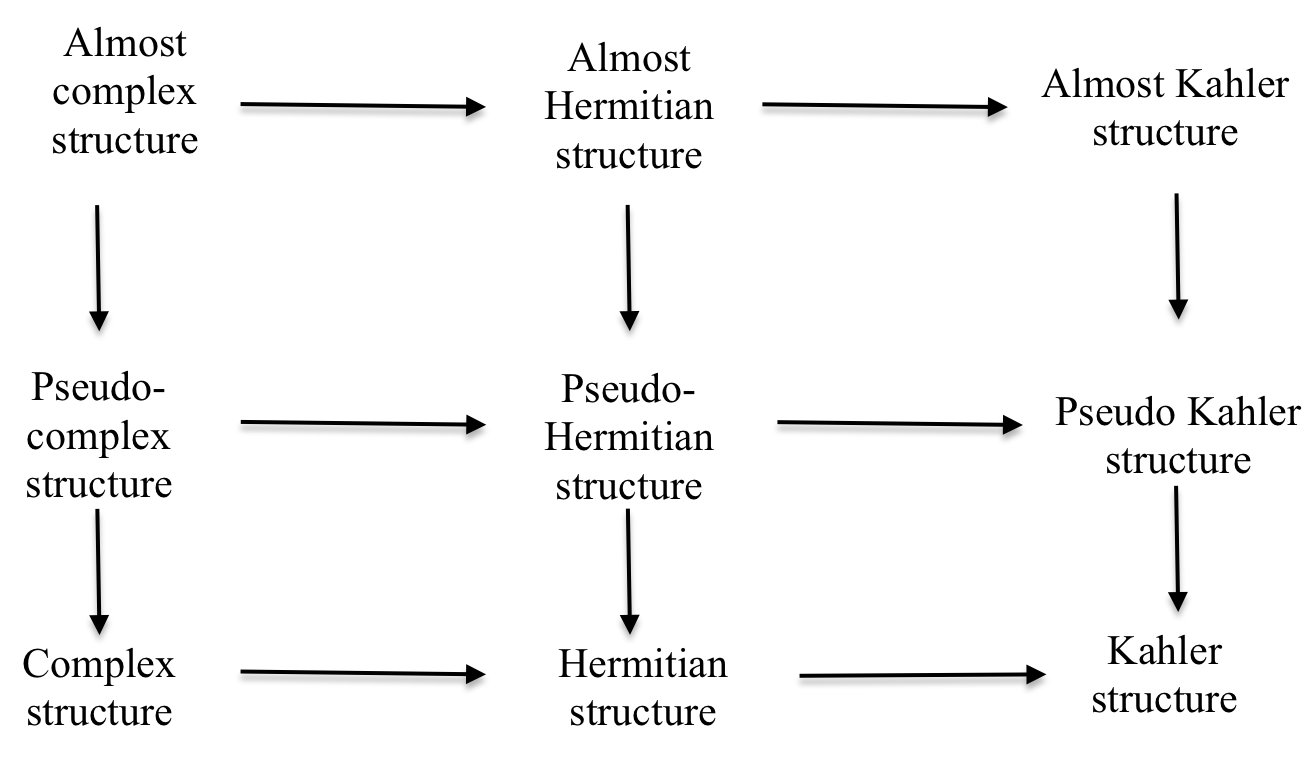}
    \caption{Relations between spaces}
    \label{fig:relations}
\end{figure}

 \subsection{Hermitian WDVV equation}\label{S:HWDVV}
Since we would like to consider complex manifolds, it is necessary to generalise the WDVV non linear PD equation for complex structures. We introduce the  Hermitian WDVV equation, for the convenience of the reader, as follows. 
\begin{dfn}
 Assume $(M,g)$ is a compact Hermitian  manifold. The Hermitian WDVV equation is satisfied if the following holds: 
    \begin{equation}\label{E:HWDVV}
        \forall a,b,\bar{c},\bar{d}, \quad \sum_{\bar{e}f}\Phi_{ab\bar{e}}g^{\bar{e}f}\Phi_{f\bar{c}\bar{d}}=\sum_{\bar{e}f} \Phi_{b\bar{c}\bar{e}}g^{\bar{e}f}\Phi_{fa\bar{d}},
    \end{equation}
where:

\smallskip 
\begin{itemize}

   \item$\Phi$ is a real potential function on $M$
   \item[] 
   
   \item $ \Phi_{ij\bar{k}}=\frac{\partial^3\Phi}{\partial{z}^{i}\partial{z}^{j}\partial \bar{z}^k}$ in local coordinates; 
   \item[]
   
   \item $g$ is a non-degenerate K\"ahler metric on $M$ and  we have $(g^{\bar{e}f})=(g_{\bar{e}f})^{-1}$.
\end{itemize}
\end{dfn}

\subsection{Frobenius manifolds: prerequisites, recollections, notations}\label{S:1.2}

\medskip 

We refer to \cite{Man1,Man2} for a definition of (pre-)Frobenius manifolds. We  introduce/ recall important prerequisites and notions. 


 \subsubsection{}

 {\bf Man} is the category of analytic manifolds over a field $\K$ (being $\R$ or $\C$). 
 Let $M \in Ob({\bf Man})$ be an object of {\bf Man}. Denote by $\cO_M$ (resp. $\Omega^1_M$) the sheaf of holomorphic (resp. holomorphic 1-forms) on $M$. 

\smallskip 

 \subsubsection{Affine connections} 

\smallskip 

 Assume $\mathscr{E}$ is the space of sections of the tangent bundle over $M$ (in a more general context it can be a quasi coherent sheaf of $\cO_M$-modules). 

\smallskip 

 A connection on $\mathscr{E}$  is a homomorphism $\nabla$ of abelian sheaves $\nabla: \mathscr{E}\to \Omega^1\otimes \mathscr{E}$ such that the Leibnitz rule is satisfied. That is $$ \nabla(fe)=f\nabla(e)+df\otimes e,$$ where $f$ is a smooth function and $e$ is a section of $\mathscr{E}$ defined over an open subset in $M$.

\smallskip 

 Consider the class of affine (compact) manifolds.

\smallskip 

\subsubsection{Affine structures}

\smallskip 

 An affine structure on an $n$-dimensional manifold $M$ is defined by a collection of coordinate charts $\{U_a,\phi_a\}$, where $\{U_a\}$ is an open cover of $M$ and $\phi_a:U_a\to \R^n$  is a local coordinate system such that the coordinate change $\phi_b\circ \phi_a^{-1}$ is an affine transformation of $\phi_a(U_a\cap U_b)$ onto $\phi_b(U_a\cap U_b)$. 

\smallskip 

An affine structure on $M$ induces a flat and torsionless affine connection $\nabla_0$ on $M$.

\subsubsection{Affine manifold}
An affine manifold is a manifold with an affine structure. The fundamental group of a compact complete flat affine manifold is an affine crystallographic group. 
\smallskip

\subsubsection{Crystallographic groups}

A  group $\Lambda$ is an $n$-crystallographic if $\Lambda$ contains a normal, torsion-free, maximal abelian subgroup of rank $n$ and finite index. A crystallographic group satisfies the short exact sequence 
\[0\longrightarrow V \longrightarrow \Lambda \longrightarrow P\longrightarrow 1,\] with 
$P\leq GL(n,\Z)\cong Aut(V)$ is a finite group acting faithfully on $V$. A complex crystallographic group arises as a discrete subgroup $\Lambda \subset Iso(\C^n)$ such that $\C^n/\Lambda$ is compact, where  $Iso(\C^n)$ is the group of biholomorphisms in $\C^n$ preserving the standard Hermitian metric.  The class of crystallographic groups being torsion-free are called Bieberbach groups.

\medskip 

\subsubsection{}
In the language of vector bundles, a manifold (over $\K$) with an affine structure is equipped with a tangent bundle, where the $G$-structure is given by the group of affine transformations $Aff(n)=GL(n,\K)\rtimes \K^n$ and $GL(n,\K)$ is the general linear group over the field $\K$. 

\medskip

\subsubsection{Metrics} Let $(\mathscr{E},M,\pi)$ be a fiber bundle over $M$ with structure group $G$. 
Assume that $M$ and the fiber space are equipped with a (non degenerate) bilinear symmetric form.  
For each $x\in M$, there exists a Hermitian (resp. Riemannian) inner product on the fiber $\mathscr{E}_x$ defined as:
 $$\langle -,- \rangle:\, \mathscr{E}_x\times \mathscr{E}_x\to \C.$$
For any open in $M$ and sections $\xi, \eta$ in $\mathscr{E}$ the mapping $x\mapsto  \langle \xi,\eta \rangle_x$ is differentiable.

\subsection{Frobenius Algebra bundles}\label{S:1.3} We introduce Frobenius algebra bundles.

\subsubsection{Algebras}
Assume $(\A,\circ)$ is an algebra of rank $n$. Then, the multiplication operation: 
$$\circ:\A\times \A\to \A \quad \text{is defined by}\quad e_i\circ e_j=\sum_{k} C_{ij}^k e_k,$$ where  $C_{ij}^k$  are the structure constants of $\A$ and $e_i$ are generators of $A$. 
\begin{itemize}

    \item[---]  The algebra $\A$ is {\it commutative} if the commutator vanishes:
    \begin{equation}\label{E:Com}
       \bK_{\alpha \beta}^\gamma=C_{\alpha \beta}^\gamma-C_{\beta\alpha }^\gamma=0.
    \end{equation}

    \item[---]   The algebra $\A$ is {\it associative} if the associator vanishes:

\begin{equation}\label{E:Ass}
    \bA_{\alpha \beta \gamma}^b=C_{\alpha \beta}^\delta C_{\delta \gamma}^b-C_{\alpha\beta}^b C_{\beta\gamma}^\delta=0.
\end{equation}

\end{itemize}
\medskip 
\subsubsection{Frobenius algebra} 

A {\it Frobenius algebra} $(\A,\circ)$ over $\K$ is an associative, commutative, unital algebra with multiplication operation $``\circ"$
 equipped with a symmetric bilinear form $\langle-,-\rangle$ satisfying 
 \begin{equation}\label{E:bili}
     \langle x\circ y, z\rangle=\langle x,y\circ z\rangle,\quad  \forall x,y,z \in \A.
 \end{equation}
\medskip 

\subsubsection{Pre-Frobenius algebra} 

A {\it pre-Frobenius algebra} is an algebra $(\A,\circ)$  satisfying commutativity (i.e. Equation \ref{E:Com}) and Equation  \ref{E:bili}. It is not a Frobenius algebra because it is not associative. So, we have the table in Fig.\ref{fig:pre+Frobenius}.

\medskip

\begin{figure}[ht]
    \centering
\includegraphics[scale=0.6]{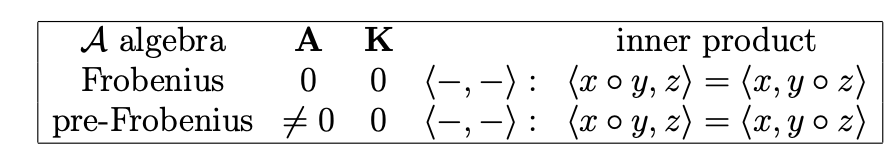}
    \caption{pre-Frobenius versus Frobenius algebras}
    \label{fig:pre+Frobenius}
\end{figure}


\subsubsection{Algebra bundle}

 A {\it bundle algebra} with base $M$ is a vector bundle ${\pazocal{E}}=(\sE(\A),M,\pi)$ endowed with a bilinear morphism such that each fiber $\sE_x=\pi^{-1}(x)$ for any $x\in M$ is endowed with the structure of the algebra $\A$. Given $x\in M$, if the algebra $\A_x$ has an identity element $\epsilon_x$ then the mapping  $x\mapsto \epsilon_x$ is a unit section of $\sE$. 

\medskip 

 Algebra bundles can be obtained as a result of left(/-right) regular representations of an algebra of rank $n$ in a typical fiber of a vector bundle $\pazocal{E}$. 
In this context, the structure group $Gl_n$ of the bundle $\pazocal{E}$ is reduced to the {\it subgroup} $\tilde{G}$ of matrices commuting with the structural affinors of the algebra $\A$. The latter is denoted $\underset{b}{\hat{S}}=(C^\alpha_{b\beta})$ and $C^\alpha_{b\beta}$ are structure constants of $\A$.

 It is assumed that in the adjoint principal bundle $P(M,\tilde{G})$ of admissible frames $R_x$ the matrices $\hat{S}_b$ (and thus structure constants of $\A$) are constant. For instance they have a normal Jordan form. 

The connection 1-forms $\omega_\beta^\alpha$ on this bundle take values in the Lie algebra corresponding to $\tilde{G}$.
We have: \[\omega_\beta^\alpha = \lambda^b(x^s)\underset{b}{R}_{\beta}^\alpha=\lambda^b(x^s)C^\alpha_{b\beta} ,\]
where  $\lambda^b(x^s)\in \Omega^1_M$ are linear forms on $M$ and $\underset{b}{R}_{\beta}^\alpha$ are affinors given by structure constants.
The connection coefficients are given by

\[\Gamma_{ib}^{\alpha}=\lambda_i^b(x^s)C_{\beta b}^{\alpha}\] where $\lambda_i^b(x^s)\in \Omega^1_M$ and $C_{\beta b}^{\alpha}$ are structure constants

\medskip



 \subsubsection{Frobenius algebra bundle}
Let $M$ be a Riemannian manifold, smooth and finite dimensional. 
Assume $\pazocal{E}=(TM,M,\pi)$ is the tangent vector bundle. Denote by $Aff(n)$ the group of affine transformations. 

\begin{dfn} Let $M$ be an $n$-dimensional manifold.
A Frobenius (algebra) bundle $\pazocal{E}$ is an algebra bundle with the following conditions:

\smallskip 

\begin{enumerate}
   \item[---] $\pazocal{E}$ is a flat bundle with structure group $Aff(n)$, equipped with an affine flat torsionless connection $\con{0}$.
      \item[]
    \item[---] $\pazocal{E}$ is a Frobenius algebra bundle, where for every $x\in M$ each fiber $\sE_x$ is a Frobenius algebra $(A_x,\circ)$.
   \item[]
   
     \item[---] For every $x\in M$ there exists on the fiber $\sE_x$ a Riemannian inner product
 $$g=\langle -,- \rangle:\, \sE_x\times \sE_x\to \R.$$ 


     \item[]

    \item[---] There exists everywhere locally on $M$ a real valued potential function $\Phi$. There exists rank three symmetric tensor 
    given in local coordinates by $C_{ijk}=\partial_i\partial_j\partial_k\Phi$ where $\partial_l=\frac{ \partial}{ \partial x^l}$.
    
   \item[]

    \item[---] Structure constants of the algebra $\A_x$ are given by relations: 
    $$C_{ij}^k=g^{ek}C_{eij},$$
    where $(g_{ek})^{-1}=(g^{ek})$.

\end{enumerate}
\end{dfn}

\medskip 
\begin{rem}
Assume $M$ is an $n$-dimensional manifold.
The manifold $M$ comes equipped with a pre-Frobenius bundle $\pazocal{E}$ if the algebra bundle is a pre-Frobenius algebra bundle.

\smallskip

    The important difference between a pre-Frobenius bundle and a Frobenius bundle relies on the structure of the algebra $\A_x$:
    how  structure constants are defined and on the associator $\bA$. To summarise:
\[
\begin{cases}
 \forall x\in M,\quad \text{if}\quad \bA =0 & \text{then}\quad \sE_x\quad \text{is a Frobenius algebra}; \\

\smallskip

 \forall x\in M,\quad \text{if}\quad \bA \neq 0 & \text{then}\quad  \sE_x \quad \text{is a pre-Frobenius algebra}.
\end{cases}\]

\end{rem}

\section{Proof (of Theorem \ref{T:equ})}\label{S2:Proof}
This section is devoted to proving Theorem \ref{T:equ}.

For simplicity, we separate proofs for the real and complex cases. 
\subsection{Real case}\label{S2:realcase}
We omit the proof of equivalence between (1) and (2) in the real case, since this follows from the literature \cite{Man1}. 

\,

We prove the equivalence between statement (2) and statement (3).
\begin{lem}\label{L:2-3R}
Let $M$ be a Hessian manifold (finite dimensional). Then the following are equivalent:
\begin{itemize}
    \item[(2)] $M$ is a Frobenius manifold;
    \item[(3)] $M$ is equipped with a Frobenius bundle.
\end{itemize} 
\end{lem}

Before the beginning of the proof we introduce a few notations. $TM$ stands for the tangent bundle; $T_M$ for the tangent sheaf; $\Gamma(TM)$ are sections of the tangent bundle. 

\, 
\begin{proof}
    
A central ingredient of the proof resides in the notion of pre-Frobenius manifolds. 

\, 
In order to define a pre-Frobenius manifold, it is necessary at first to have an affine manifold $M$. This manifold is equipped with a flat affine connection $\con{0}$. 

\, 

From this basic assumption, we can obtain flat coordinates i.e. coordinates verifying $\con{0}(dx^i)=0$. One defines a metric $g$, compatible with $\con{0}$.  

\, 

Assuming the potentiality axiom on $M$ holds (this states that everywhere locally on $M$ there exists a smooth potential function $\Phi$) one can define a metric tensor which, in those flat vector fields, is a Hessian (resp. K\"ahler) type of metric. Another important component is the existence of a rank three symmetric tensor, which in the flat vector fields is obtained by taking the partial derivatives of $\Phi$ (to the third order). 

\, 

The differential geometric setting allows us to define naturally a multiplication operation $\circ$ on the tangent sheaf. By construction, this algebra is unital and commutative. Given flat vector fields, the constant structures of the algebra are induced by  Christoffel symbols. The final question concerning the associativity of the algebra turns out to be crucial since the pre-Frobenius manifold is a Frobenius manifold if the algebra is associative. 

\, 

More formally, we express the above statements in the following list. 
\begin{enumerate}
 \item[] 
 
\item $(M,\con{0})$ is a Riemannian manifold equipped with an affine structure, with $\con{0}$ being a flat affine  torsionless connection on $T_M$, given by $$\con{0}:T_M\to \Omega^1_{M}\otimes_{M}T_{M},$$ where $\Omega_{M}^1$ is the sheaf of 1-forms on ${M}$.
   \item[] 
    \item There exists a non-degenerate symmetric bilinear form on the tangent bundle $g$, where $g$ is compatible with the connection $\con{0}$. 
       \item[] 
       
    \item  There exists a rank three symmetric tensor $A$ satisfying for flat vector fields $X,Y,Z$ the relation $A(X,Y,Z)=g(X\circ Y,Z)=g(X,Y\circ Z)$.

\medskip 
 
The multiplication law $\circ$ on $T_{M}$ is given by: \[X\circ Y:= i_X(\widehat{A})(Y),\] 
where $\widehat{A}$ is a global section lying in $\widehat{A}\in \Omega^1_{M}\otimes_{\cO_{M}}End(T_{M})$. For an element in the endomorphism $\tilde{F}\in End(T_{M})$ and $df\in\Omega_{M}^1$ ($f$ is a smooth function) we get $i_X(df\otimes G):= Xf \cdot \tilde{F}.$ More specifically, in local flat coordinates the multiplication operation is defined as  
\begin{equation}\label{E:circ}
 (\star)\quad   \partial_a\circ \partial_b =\sum_{c}A_{ab}^c\partial_c,
\end{equation}
where $A_{ab}^c=\sum_eA_{abe}g^{ec}$. 

\smallskip 

\item If everywhere locally there exists a potential function $\Phi$ such that $A(X,Y,Z)=(XYZ)\Phi$, then this pre-Frobenius manifold is potential. 
\item[]

In particular, using  local (flat) coordinates, denoted $(x^a)$, one has the following situation: 
$$A_{abc}=\partial_a\partial_b\partial_c \Phi$$ where $\Phi$ is a real potential function. In general, we denote it also as $A(X,Y,Z)=(XYZ)\Phi$, where $X=\frac{\partial}{\partial x^a},Y=\frac{\partial}{\partial x^b},Z=\frac{\partial}{\partial x^c}$ are flat vector fields.

\smallskip

\item {\it In fine} if a potential pre-Frobenius manifold satisfies everywhere the associativity condition: $(X\circ Y)\circ Z=  X\circ (Y\circ Z)$ then this is a Frobenius manifold.
\end{enumerate}

\medskip 

\, Now, we prove our statement. Let $M$ be an $n$-dimensional Riemannian manifold.

\,

 Assume (1) is true. Therefore,  $M$ is equipped with an affine structure. By definition, this implies that the transition functions are elements in $Aff(n).$ By construction of the vector (tangent) bundle, the tangent bundle $E$ is endowed with a structure group $Aff(n).$ By definition of an affine structure, this tangent bundle is equipped with a flat torsionless affine connection $\con{0}$.

\, 

 The existence of (2) equips $E$ with a symmetric bilinear form: a Riemannian metric, compatible with $\con{0}$.

\,

Let us analyse the meaning of Equation \ref{E:circ}.
In this Riemannian geometry framework, the multiplication operation is given by the covariant derivative $\con{0}_X(Y)$. 
Given  flat vector fields $X,Y$ the multiplication is obained by putting $\con{0}_X(Y):=X\circ Y.$
Objects $A_{ij}^k$ in this Riemannian context are identified with the Christoffel symbols $\Gamma_{ij}^k.$ 

\,

 If (3) is satisfied, then at each point $x\in M$ there are sections $s_1(x),\cdots, s_n(x) \in \Gamma(TM)$ in the fiber such that $s_i(x)\circ s_j(x)=\sum\limits_k A_{ij}^ks_k(x)$, where 
$A_{ij}^k=A_{ije}g^{ek}$ and $\circ$ is defined as above. The existence of a Riemannian metric on $TM$ provides $TM$ with a bilinear symmetric form $\langle-,- \rangle$.

\, 

Assume  (1), (2) and (4) is true. This  combination implies that there exists a  metric, compatible with $\con{0}$. By the construction of a Frobenius manifold it is necessarily a Hessian metric $g$ on $TM$. Everywhere locally there exists a real potential function $\Phi:M \to \R$ such that in local coordinates the metric is $g_{ij}=\partial_i\partial_j\Phi$. 

\,

(5) is true if and only if the affine space of connections 
$\con{\lambda}$ is flat, where 
$\con{\lambda}$ is a pencil of connections depending on an even parameter such that 
$\con{\lambda}: T_M\to \Omega^1_M\otimes T_M$ and $$\con{\lambda}:= \con{0} +\lambda (X\circ Y).$$

\smallskip

It is a routine exercise to show that the flatness of the pencil of connection is equivalent to having associativity on the level of the vector fields. In other words, if (5) is true, then one has an algebra bundle, where each fibre is isomorphic to an $n-$dimensional Frobenius algebra.  

\end{proof}

\subsection{Complex case}\label{S:HermitianFman}
We generalise the previous constructions to the complex case.

 \medskip 
Assume $M$ is a complex manifold. Then, there exists a positive definite Hermitian metric on $M$. The metric is given by 
 \begin{equation}
     ds^2=g_{a {\bar b}}(z,\overline{z})dz^ad\overline{z}^b. 
 \end{equation}
 The Hermitian metric over $M$ is naturally a hermitian metric on the tangent bundle. 

 \medskip 

\subsubsection{}\label{S:Kahler} The existence of a {\it Hermitian} WDVV equation requires the following necessary (but not sufficient) conditions:

\smallskip 

\begin{itemize}
    \item[---] $(M,g)$ is a complex manifold with holomorphic affine structure and hermitian metric $g$. 

 \item[]
 
     \item[---] There exists on $M$ a  Levi--Civita connection $\con{0}$.  

      \item[]

      \item[---] Everywhere locally there exists a potential function $\Phi$ such that the metric $g$ can be expressed as $g_{a {\bar b}}=\frac{\partial^2\Phi(z,\bar{z})}{\partial z^a \partial \bar{z}^b}$. 

      \item[]

       \item[---] There exists a rank three symmetric  tensor which can be expressed (modulo a suitable change of coordinates) as $A=(XY\bar{Z})\Phi$. 

       \item[]

\end{itemize}

\subsubsection{K\"ahler geometry}\label{S:Kahler2}
We prove that:

\begin{lem}\label{L:Kahler}
 Let $(M,g)$ be a compact hermitian manifold. Assume $M$ satisfies the conditions in Section \ref{S:Kahler}. Then, $(M,g,A)$ is necessarily a K\"ahler manifold.
\end{lem}

 \medskip
 \begin{proof}
 By Schouten--VanDanzig if $M$ admits a Levi--Civita connection then the metric tensor necessarily have the following shape $g_{a {\bar b}}=\frac{\partial^2\Phi(z,\bar{z})}{\partial z^a \partial \bar{z}^b},$ where $\Phi(z,\bar{z})$ is a real valued function.

  \medskip

Let the corresponding (1,1) form be $\omega=\frac{\imath}{2}g_{a {\bar b}}(z,\overline{z})dz^a\wedge d\overline{z}^b$. This construction extends to a K\"ahlerian metric if K\"ahler's conditions are satisfied. The K\"ahler condition requires the local existence of a K\"ahler potential $\Phi$ in any coordinate system
   (unique modulo a pluriharmonic function) such that:
   
    \begin{equation}
        g_{a {\bar b}}=\frac{\partial^2\Phi(z,\bar{z})}{\partial z^a \partial \bar{z}^b}.
    \end{equation}

  \medskip

  The K\"ahler condition also requires that the 2-form 
  \begin{equation}\label{E:1}\omega=g_{a {\bar b}}(z,\overline{z})dz^ad\overline{z}^b\end{equation} is exact.
 
   \smallskip 
   
 In other words,  
 \begin{equation}\label{E:2}
   \forall a,b,c: \quad \frac{\partial g_{b\overline{c}}}{\partial z^a}=\frac{\partial g_{a\overline{c}}}{\partial z^b},\quad \frac{\partial g_{a\overline{b}}}{\partial \overline{z}^c} =\frac{\partial g_{a\overline{c}}}{\partial \overline{z}^b}. \end{equation}

 \medskip

Assume $M$ is as above. Then, locally everywhere on the manifold, there exists a rank three  symmetric tensor:
$$\frac{\partial g_{b\overline{c}}}{\partial z^a}=\frac{\partial^3 \Phi}{\partial z^a\partial z^b\partial \bar{z}^{c}}$$

 Denote this object, for simplicity as: $\partial_a \partial_b\partial_{\overline{c}}\Phi=\Phi_{ab\bar{c}}.$

 \smallskip

By K\"ahler's conditions, the following relations are satisfied: 
 \begin{equation}\label{E:rk3}
 \Phi_{ab\bar{c}}=\Phi_{ba\bar{c}} \quad\text{and}\quad \Phi_{\bar{c}a\bar{b}}=\Phi_{\bar{b}a\bar{c}}
 \end{equation}

\medskip

 A K\"ahler metric satisfies the properties that: $g_{i\bar{j}}$ is symmetric and $g_{\bar{i}\bar{j}}=g_{ij}=0$. 

 Therefore, the following tensors vanish:  $\Phi_{\bar{c}ab}=0$ (resp. $\Phi_{c\bar{a}\bar{b}}=0)$.
  
  Assuming we ignore the cases where the metric vanishes by definition of $g_{ij}$, the tensors $\Phi_{ab\bar{c}}$ (resp. $\Phi_{\bar{a}b\bar{c}}$) can be considered as rank three symmetric tensors.

  \smallskip 

To conclude, one can easily see that if the metric is not K\"ahler then not all of the required properties can be satisfied. 
\end{proof}

\subsubsection{}
K\"ahler metrics are naturally present on bounded domains (this is a classic result by Kobayashi on bounded domains \cite{Ko}).
This follows from the existence of a well defined Bergmann kernel on bounded strongly pseudo convex domains (also known as Stein manifolds). However, there exist also manifolds admitting a Bergmann metric which are not necessarily bounded manifolds.
 \footnote{On such manifolds, there exists a 
strongly pseudo convex exhaustion function. A real function $\psi: M \to \R$ with $\imath \partial \overline{\partial}\psi>0$ and such that $M_c=\{ z\in M \, |\, \psi(z)\leq c\}$.} 

 \medskip 
 
\subsubsection{Hermitian Frobenius manifold}
We give the following definition. 

\begin{dfn}
A Hermitian Frobenius manifold $M$ is a complex manifold with:

\begin{itemize}
  \item[]
  
    \item[---] 
holomorphic affine structure; 
  \item[]
 \item[---] a flat affine torsionless connection $\con{0}$; 
   \item[]
        \item[---] a K\"ahler metric $g$, compatible with the connection; 
          \item[]
            \item[---] a commutative, associative multiplication operation $\circ$ on the tangent sheaf, induced by the covariant derivative $\con{0}_X(Y):=X\circ Y$ for given flat vector fields. 
              \item[]
                \item[---] A rank three symmetric tensor $A$ such that: $A(X,Y,Z)=g(X\circ Y, Z)=g(X,Y\circ Z)$ for given flat vector fields.
\end{itemize}

\end{dfn}
\medskip

\subsubsection{Proof of Theorem \ref{T:equ}: (1) is equivalent to (2)}
\begin{lem}\label{L:1to2}
Let $(M,g)$ be a K\"ahler manifold. Then, the following statements are equivalent:
\begin{enumerate}

\item the Hermitian WDVV equation is satisfied;
    \item $M$ is a Hermitian Frobenius manifold;  
   \end{enumerate}
\end{lem}
\begin{proof}
The proof relies on the the previous subsections and follows from the deep relation between the associativity equations in the Frobenius algebra and the flatness of the curvature tensor.

The curvature tensor $R_{X,Y}(Z)$ obeys to the following:  

\begin{equation}\label{E:C}
 R_{X,Y}(Z)= X \circ (Y \circ Z)-Y \circ (X \circ Z)
\end{equation}

given vector fields $X,Y,Z\in T_M^{1,0}$ (or respectively $X,Y,Z\in T_M^{0,1}$).

Therefore, $R_{X,Y}(Z)=0$ if and only if the associator $\bA=0$ i.e. $X \circ (Y \circ Z)=Y \circ (X \circ Z).$

\medskip 

By commutativity of the algebra we get: $X \circ (Y \circ Z)=Y \circ (Z \circ X)$.

Iterating Equation \ref{E:C} gives:  $Y \circ (Z \circ X)=Z \circ(Y \circ X)$.

Therefore, $X \circ (Y \circ Z)= Z \circ(Y \circ X)$. 

By commutativity:
\begin{equation}\label{E:asso}
X \circ (Y \circ Z)=(X \circ Y)\circ Z.    
\end{equation} 

\smallskip 
 
In local coordinates, introducing $X=\partial a, Y=\partial_b, Z=\partial_c$ into Equation  \ref{E:asso}  one has:
\begin{equation}\label{E:Ass}
\nabla_{\partial a}(\, \partial_b \, \circ\,  \partial_c )= \nabla_{\partial b}(\, \partial_a\, \circ \, \partial_c).
\end{equation}

Now, in local coordinates, the curvature tensor is defined as follows. 
\begin{equation}\label{E:R}
 R_{a \bar{b} c \bar{d}}=\frac{\partial^2 g_{a\bar{b}}}{\partial z^c\partial \bar{z}^d}-g^{e\bar{\gamma}} \frac{\partial g_{a\bar{\gamma}}}{\partial z^c}\frac{\partial g_{e\bar{b}}}{\partial \bar{z}^d}.\end{equation}

If the curvature tensor is 0 then $R_{a \bar{b} c \bar{d}}=0$ implies that
 
 $$\frac{\partial^2 g_{a\bar{b}}}{\partial z^c\partial \bar{z}^d}=g^{e\bar{\gamma}} \frac{\partial g_{a\bar{\gamma}}}{\partial z^c}\frac{\partial g_{e\bar{b}} }{\partial \bar{z}^d}.$$
From computations, we obtain the hermitian WDVV equation such as as defined in Equation  \ref{E:WDVV}. 

The hermitian WDVV equation is satisfied if and only if at each point on $M$, the tangent space has the structure of Frobenius algebra.

\end{proof}

\medskip 

\subsubsection{Hermitian Frobenius bundles}
\label{S:HFB}
For the complex setting, a reformulation of the construction of Frobenius bundles is presented.

\smallskip 
\begin{dfn} Assume $M$ is an $n$-dimensional hermitian manifold over $\C$.
The manifold $M$ comes equipped with a Frobenius bundle $\sE$ if the following are all satisfied:

\smallskip 

\begin{enumerate}
 \item[---] $\sE$ is a Frobenius algebra bundle, where each fiber is a complex $n$-dimensional Frobenius algebra over $\C$.

   \item[]
   
    \item[---] $\sE$ is a flat bundle with structure group $Aff(n)$ equipped with an affine flat torsionless connection $\nabla_0$. 
   
     \item[]
     \item[---] For every $x\in M$ there exists on the fiber $E_x$ a positively defined K\"ahler form.

   \item[]

    \item[---] Everywhere locally, there exists a rank three symmetric tensor $C_{ki\bar{j}}:=\partial_{k} g_{i\bar{j}}$, where $g_{i\bar{j}}$ is the K\"ahler metric. Then, $A_{ki\bar{j}}$ is equal to $g(s_i\circ s_{\bar{j}},s_k)=g(s_i,  s_{\bar{j}} \circ s_k)$, where $s_i \in \Gamma(TM)$ 
are sections of the tangent bundle. 
\end{enumerate}
\end{dfn}

This construction allows us to prove the second part of Theorem \ref{T:equ}.

\medskip 

\begin{lem}\label{L:2-3C}
Assume $M$ is a compact complex manifold. The following are equivalent: 
\begin{enumerate}
    \item $M$ is a Hermitian Frobenius manifold. 
    \item $M$ is equipped with a Frobenius bundle. 
\end{enumerate} 
\end{lem}
\begin{proof}
The proof mimics essentially the same steps as for the real case (see Lem. \ref{L:2-3R}). However, there is a  difference relying on the constructions described in Section \ref{S:HermitianFman} (precisely, sections \ref{S:Kahler}--\ref{S:Kahler2}).  
\end{proof}

\medskip 

\section{K\"ahler--Frobenius manifolds}\label{S:KFmfd}
 {\it Are there K\"ahler compact manifolds which satisfy the axioms of a Frobenius manifolds?}  We give a positive answer to this question and give a classification of such manifolds.

\subsection{}
We show that a complex K\"ahler manifold with vanishing curvature admits the structure of a Frobenius manifold.  

\medskip 

\begin{Theorem}\label{T:Main} 
Let $M$ be a compact complex K\"ahler manifold with vanishing curvature. Then, $M$ is a hermitian Frobenius manifold.
\end{Theorem}
In particular for simplicity a K\"ahler manifold carrying the structure of a  hermitian Frobenius manifold is referred as K\"ahler--Frobenius manifold.

\,

\begin{proof}~

We first show that if $M$ is a compact complex K\"ahler manifold with vanishing curvature then $M$ admits a holomorphic affine connection. 

\, 

1) The manifold $M$ admits a holomorphic affine connection 
if and only if  the 1-cocycle defining an element of 
$H^1(M,\Omega^1\otimes\Omega^1\otimes\cO)$ is zero.

\,

This  1-cocycle is cohomologous to zero 
if and only 
if there exists a 0-coboundary 
in $H^0(U,\Omega^1\otimes\Omega^1\otimes\cO)$ 
such that 
$C_{VU}=b_V-b_U$, where 

\[b_U=\sum \Gamma_{Ujk}^i d u^j\otimes d u^k\otimes \frac{\partial}{\partial u^i} \]
 then $C_{VU}=b_V-b_U$ is the transformation law for Christoffel symbols 
 \[\Gamma_{Ujk}^i-\sum \frac{\partial u^i}{\partial v^a}\Gamma_{Vbc}^a \frac{\partial v^b}{\partial u^k}\frac{\partial v^c}{\partial u^k}=
 \sum\frac{\partial u^i}{\partial v^a}\frac{\partial^2 v^a}{\partial u^j\partial u^k}.\]

\,

Therefore, this proves that if $M$ a K\"ahler manifold with vanishing curvature then $M$ admits a holomorphic affine connection.

\,

2) If $M$ is equipped with a holomorphic affine structure then there exists a flat, torsionless affine connection on $M$. Let $\con{0}$ be this connection.
Let $\sA$ be the (affine) space of connections on $M$. For all $x \in M$ and any connection $\tilde{\nabla}\in \sA$, 
the difference $$\tilde{\nabla}-\nabla$$ is given by a smooth bilinear bundle homeomorphism 
$$\tilde{\Gamma}:TM\times TM \to TM,$$ where  $\tilde{\Gamma}$ depends smoothly on $x\in M$. 

\, 

Reciprocally, given a connection $\nabla \in \sA$ and $\tilde{\Gamma}$ 
there exists an affine connection in $\tilde{\nabla}\in \sA$ corresponding to $\nabla + \tilde{\Gamma}$.
Since $\sA$ is an affine space, it is enough to vary $\con{0}$ along lines as follows: 
$$\con{\lambda}=\con{0}+\lambda \eta$$ where $\eta$ is a given 1-form and $\lambda$ is a real parameter.

\,

 \,
 
3) This construction allows to define a multiplication operation on the tangent sheaf. We get: 
 $$X\circ Y:=i_X(\tilde{\Gamma})(Y),$$ where 
 $$i_X(df\otimes Q)=Xf\cdot Q, $$ for objects
 \begin{itemize}
     \item[---] $f\in O_M$,

     \item[---] $df\in \Omega^1_M$;

      \item[---] $Q\in End(TM)$.
 \end{itemize} 
 In local coordinates, one obtains: $$\partial_a \circ \partial_b =\sum_c\Gamma_{ab}^{c}\partial_c.$$
 
 \,
 
4) The complex tangent space $TM$ decomposes into a holomorphic and anti-holomorphic part:
$$TM=T^{1,0}_M\oplus T^{0,1}_M.$$
By construction of a K\"ahler metric \cite[Chap. 8]{BY} the following Christoffel symbols vanish:  
 $\Gamma_{a\bar{b}}^c=0$,  $\Gamma_{\bar{a}\bar{b}}^c=0$ and $\Gamma_{ab}^{\bar{c}}=0$. 
 Therefore, the non-null Christoffel symbols are: $\Gamma_{ab}^c$ and  $\Gamma_{\bar{a}\bar{b}}^{\bar{c}}$.
 
\,
 
5) We prove now that there exists a pre-Frobenius algebra structure on $(T^{1,0}_M,\circ)$ (resp. on $(T^{0,1}_M,\circ)$), where $\circ$ is defined above.
In local coordinates, this is illustrates as:  
  $$\partial_a \circ \partial_b =\sum_c\Gamma_{ab}^{c}\partial_c$$ where $\partial_a,\partial_b, \partial_c \in T^{1,0}_M$ 
  respectively:  $$\partial_{\bar{a}} \circ \partial_{\bar{b}} =\sum_{\overline{c}}\Gamma_{\bar{a}\bar{b}}^{\bar{c}}\partial_{\bar{c}},$$
 where $\partial_{\bar{a}},\partial_{\bar{b}}, \partial_{\bar{c}} \in T^{0,1}_M$. 
  Since there exists a bilinear symmetric form given from the K\"ahler metric on $E_x$, 
 these algebras inherit naturally a bilinear symmetric form.

 \,
 
By definition of a K\"ahler metric  
 $g_{ab}=0$ and $g_{\bar{a}\bar{b}}=0$. 
 Therefore, on $(T^{1,0}_M,\circ)$ the symmetric bilinear form satisfies:
  
$$g(\partial_a\circ \partial_b,\partial_c)=g(\partial_a, \partial_b\circ\partial_c)= \partial_a \partial_{b}\partial_c \Phi=0.$$
 
Similarly,  on  ($T^{0,1}_M,\circ$) we get  $g(\partial_{\bar{a}}\circ \partial_{\bar{b}},\partial_{\bar{c}})=g(\partial_{\bar{a}}, \partial_{\bar{b}}\circ\partial_{\bar{c}})= \partial_{\bar{a}} \partial_{\bar{b}}\partial_{\bar{c}} \Phi=0$.

 Therefore,  $(T^{1,0}_M,\circ)$ (resp. $T^{0,1}_M$) are pre-Frobenius algebras where 
 $\langle X\circ Y, Z \rangle =  \langle X, Y \circ Z\rangle=0, $ for $X,Y,Z \in T^{1,0}_M$ (resp. $T^{0,1}_M$).

\, 

 The direct sum of Frobenius algebras forms a Frobenius algebra. So, $(T_M,\circ)$ carries the structure of a Frobenius algebra.

\,

6) Everywhere locally there exists a rank three tensor on $M$:
\[\forall a,b,c,\, \quad \Phi_{ca\bar{b}}=\partial_c\partial_a\partial_{\bar{b}}\Phi \quad \text{or}\quad 
\Phi_{\bar{c}a\bar{b}}=\partial_{\bar{c}}\partial_a\partial_{\bar{b}}\Phi.\]

\,

By K\"ahler's condition: \[\forall a,b,c,\quad \Phi_{ca\bar{b}}=\Phi_{ac\bar{b}}\,\quad  \text{and}\,\quad\Phi_{\bar{c}a\bar{b}}=\Phi_{\bar{b}a\bar{c}}.\]
By symmetry of the metric tensor: 

$$\forall a,b,c, \quad \Phi_{ca\bar{b}}=\Phi_{ac\bar{b}}=\Phi_{a\bar{b}c}=\Phi_{c\bar{b}a}$$ and 
$$\Phi_{\bar{c}a\bar{b}}=\Phi_{\bar{b}a\bar{c}}=\Phi_{\bar{b}\bar{c}a}=\Phi_{\bar{c}\bar{b}a}.$$

\smallskip 

The K\"ahler metric tensors give $g_{ab}=g_{\bar{a}\bar{b}}=0$.
So,  $\Phi_{cab}$ and $\Phi_{c\bar{a}\bar{b}}$ vanish, since the metric tensors vanish. Therefore, on the directions where the K\"ahler metric tensor is not (by definition) zero, the rank thee symmetric tensor
given by $\Phi_{ca\bar{b}}$ and $\Phi_{\bar{c}a\bar{b}}$ are symmetric. 
 
\,

7)  We prove now that we have all ingredients to satisfy the hermitian WDVV equation. We use the Einstein summation convention in the following paragraphs for simplicity. The curvature tensor is given by

 \begin{equation}\label{E:R}
 R_{a \bar{b} c \bar{d}}=\frac{\partial^2 g_{a\bar{b}}}{\partial z^c\partial \bar{z}^d}-g^{e\bar{\gamma}} \frac{\partial g_{a\bar{\gamma}}}{\partial z^c}\frac{\partial g_{e\bar{b}}}{\partial \bar{z}^d}.\end{equation}

 By hypothesis, the curvature tensor vanishes i.e. $R_{a \bar{b} c \bar{d}}=0$. Therefore, Equation~\ref{E:R} is equivalent to
\begin{equation}\label{E:R1}
\frac{\partial^2 g_{a\bar{b}}}{\partial z^c\partial \bar{z}^d}= \frac{\partial g_{a\bar{f}}}{\partial z^c} g^{e\bar{f}} \frac{\partial g_{e\bar{b}}}{\partial \bar{z}^d}.\end{equation}

\smallskip 

Since $M$ is a K\"ahler manifold,  Equation\ref{E:R1} can be rewritten as  

\begin{equation}\label{E:R2}
\partial_{a}\partial_{\bar{d}}\Phi_{c\bar{b}}= \Phi_{ca\bar{f}}g^{e\bar{f}} \Phi_{\bar{d}e\bar{b}}.
\end{equation}

 By K\"ahler's condition $d\omega=0$ and by construction $d^2\omega=0$.

Therefore,
\begin{equation}\label{E:HWDVV}
\frac{\partial^2 g_{a\bar{b}}}{\partial z^c\partial \bar{z}^d}= \quad \frac{\partial g_{a\bar{f}}}{\partial z^c} g^{e\bar{f}} \frac{\partial g_{e\bar{b}}}{\partial \bar{z}^d}= \quad  \frac{\partial g_{a\bar{f}}}{\partial \bar{z}^d} g^{e\bar{f}} \frac{\partial g_{e\bar{b}}}{\partial {z}^c}\quad =  \quad  \frac{\partial^2 g_{a\bar{b}}}{\partial\bar{z}^d \partial z^c}. \end{equation}

\smallskip 

 The equation \ref{E:HWDVV} corresponds to 

\[\partial_cg_{\bar{f}a} g^{e\bar{f}} \partial_{\bar{d}}g_{\bar{b}e} = \partial_cg_{\bar{b}e}g^{e\bar{f}}  \partial_{\bar{f}}g_{\bar{d}a},\]

\smallskip 

which is a hermitian WDVV equation. 

\smallskip  

\smallskip 
 Therefore, $M$ is a hermitian (K\"ahler) Frobenius manifold. 
\end{proof}
\begin{prop}
Let $M$ be a $\C$-Frobenius manifold. Then $M$ admits also the structure of a flat smooth complex projective variety.
\end{prop}
\begin{proof}
By construction if $M$ is a $\C$-Frobenius manifold then it is smooth compact flat Riemannian manifold equipped with a K\"ahler structure. 
  By Johnsson \cite[Theorem 4.2]{John}, we have the following. Let $M$ be a smooth compact flat Riemannian manifold. If $M$ is a flat K\"ahler manifold then $M$ admits also the structure of a flat smooth complex projective variety. Therefore, the conclusion follows. 
\end{proof}
\begin{lem}
  If $M$ is a $\C$-Frobenius manifold, then  its fundamental group is an affine crystallographic group. 
\end{lem}
   
\begin{proof}
    If $M$ is a compact $\C$-Frobenius manifold then it has vanishing curvature. Therefore, it has an affine structure. The existence of an affine structure implies that $M$ is an affine manifold. So, its fundamental group is an affine crystallographic group. 
\end{proof}

\begin{cor}
  There exists a finite number of $n$-dimensional $\C$-Frobenius manifolds, for $n>1$. 
\end{cor}
\begin{proof}
    This follows from the computation of the number of affine crystallographic groups. 
\end{proof}
\subsection{Chern's Conjecture}\label{S:Chern}
Let $M$ be a holomorphic compact complex manifold. Then, expressing Chern classes in terms of the curvature of a holomorphic affine connection leads to the following fact:

\begin{lem}\label{L:Chern}
  Let $M$ be a compact complex (pre-)Frobenius manifold. Then,  all Chern classes vanish $c_i(M)=0,\quad i>0.$
\end{lem}
\begin{proof}
According to Kobayashi, we have the following:

\medskip

$$\begin{cases}
 M \text{\,has a holomorphic affine structure  then} & c_i(M)=0,\quad i\geq n/2  \\
M \text{\, has a holomorphic affine structure and is K\"ahler then} &  c_i(M)=0,\quad i>0
\end{cases}$$
   
\end{proof}

\medskip 

From this statement, it follows  that the Chern's conjecture (in affine geometry) is true for Frobenius manifolds. Chern's conjecture states that the top Chern class (Euler Characteristic) of a compact affine manifold vanishes. 

\begin{Theorem}\label{T:Chern}
The Chern conjecture is true for (pre-)Frobenius manifolds. 
\end{Theorem}
\begin{proof}
By definition, a (pre-)Frobenius manifold is an affinely flat  manifold.
By Lem~\ref{L:Chern}, a compact complex Frobenius manifold has all vanishing Chern classes. 
\end{proof}

\medskip

\section{Properties of K\"ahler-Frobenius manifolds}\label{S:Properties}

\subsection{K\"ahler--Einstein compact manifolds}\label{S:3.1}
A manifold is K\"ahler--Einstein if it admits a  K\"ahler--Einstein metric. A K\"ahler--Einstein metric implies that the Ricci form is proportional to $\tilde{\omega}$. So, there exists a constant $\dot\lambda$ such that $Ricc(\tilde{\omega})=\dot\lambda \tilde{\omega}$.

\smallskip 

If $M$ is a compact K\"ahler manifold with vanishing first Chern class $c_1(M)=0$ then in any K\"ahler class there exists a unique K\"ahler--Einstein metric with constant $\dot\lambda =0$. If the constant is $\dot\lambda =0$ then the manifold is Ricci flat. An important remark is the following. 

\smallskip

\begin{rem}
Note that if $M$ is Ricci flat then it does not imply that the manifold $M$ is flat. However, there might exist under some conditions a flat submanifold $N$ in $M.$  This is discussed in Prop~\ref{P:Ricflat} in the framework of Frobenius and pre-Frobenius manifolds. 
    
\end{rem}

\,

We prove that:

\begin{prop}\label{P:Kahler--Einstein}
Let $M$ be a compact K\"ahler--Einstein manifold with vanishing curvature. Then, $M$ is a Frobenius manifold. 
\end{prop}
 
\begin{proof}
The proof is essentially the same as for Theorem \ref{T:Main}.
\, 

Let $M$ be a compact K\"ahler--Einstein manifold. Denote the K\"ahler--Einstein metric by $g$.  By definition, there exists a K\"ahler potential $\Phi$ such that  $g_{i\bar{j}}=\partial_i \partial_{\bar{j}}\Phi$.

\,

1) \, A compact K\"ahler--Einstein manifold admits a holomorphic affine connection if and only if its curvature vanishes (see Kobayashi \cite{Ko1}). 

\, 

 Therefore, there exists a torsionless flat affine connection $\nabla_0$. 
 Assume this is a Levi--Civita connection (or Chern connection if we wish to step towards a Yang--Mills setting).
 Then, the connection $\nabla_0$ is compatible with the metric. 

\,

2) \, Define the rank three symmetric tensor $\Phi_{ijk}=\partial_i\partial_j\partial_{\bar{k}}\Phi$. 

 It satisfies Equation \ref{E:rk3}. Therefore, from our construction we have the triple $(\nabla_0,g,\Phi)$ on $M$ .

\,

3) \, We define an $\cO_M$-symmetric bilinear multiplication operation $\circ$ on the tangent sheaf $T_M$. 

 The tangent sheaf $T_M$  is locally freely generated by $(\partial_a=\partial/\partial x^a)$, where  $(\partial_a)$ determine local vector fields defined for the local coordinates $(x^a)$ in $M$.

 Therefore, there exists a multiplication operation $\circ:T_M \times T_M \to T_M$ such that in local coordinates

\[\partial_i \circ \partial_j = \sum_k\Gamma_{ij}^k\, \partial_k,  \] 

where the Christoffel symbols are defined as $\Gamma_{ij}^k:=\Phi_{ij\bar{e}}\, g^{\bar{e}k}$ and $g^{\bar{e}f}=(g_{\bar{e}f})^{-1}$.

\smallskip

 By definition of a K\"ahler metric, one has  non-vanishing Christoffel symbols: 
$$\Gamma_{ij}^k \quad \text{and} \quad \Gamma_{\bar{i}\bar{j}}^{\bar{k}}.$$

In the complex case, we decompose the tangent space as $\cT_M=\cT_{M}^{1,0}\oplus \cT_{M}^{0,1}$.

 So, we have a linear combination of two subalgberas defined on $\cT_{M}^{1,0}$ and $\cT_{M}^{0,1}$.

\,

4) \, We prove that the algebras $(T^{1,0}_M, \circ)$ and $(T^{0,1}_M, \circ)$ are associative algebras. 

Since we have that $d^2\omega=0$, \, $\sum_b dx_b \Phi_{bcd}$ is closed. So, the following is satisfied: $$\partial_a\Phi_{bcd}=\partial_{b} \Phi_{acd}.$$

\,

By hypothesis, the curvature is 0. As in Equation \ref{E:C}, the curvature tensor $R_{X,Y}(Z)$ obeys to the following:  $$ R_{X,Y}(Z)= X \circ (Y \circ Z)-Y \circ (X \circ Z)$$ given $X,Y,Z\in T_M^{1,0}$ (or respectively $X,Y,Z\in T_M^{0,1}$).

Therefore, if $R=0$ it implies that $X \circ (Y \circ Z)=Y \circ (X \circ Z).$

\,

By commutativity of the algebra we get: $X \circ (Y \circ Z)=Y \circ (Z \circ X)$.

Iterating Equation \ref{E:C} gives: $Y \circ (Z \circ X)=Z \circ(Y \circ X)$.

Therefore, $X \circ (Y \circ Z)= Z \circ(Y \circ X)$. 

Finally, by commutativity: $$ X \circ (Y \circ Z)=(X \circ Y)\circ Z. $$

\,

In local coordinates, introducing $X=\partial a, Y=\partial_b, Z=\partial_c$
into Equation \ref{E:asso} one has:
\begin{equation}
\nabla_{\partial a}(\, \partial_b \, \circ\,  \partial_c )= \nabla_{\partial b}(\, \partial_a\, \circ \, \partial_c).
\end{equation}

The left hand side of Equation \ref{E:Ass} gives us: 
$\nabla_{\partial a}(\, \partial_b\, \circ\, \partial_c\, )=\nabla_{\partial a}(\, \Gamma_{bc}^k\partial_k\, )$
$= \Gamma_{bk}^l\, \Gamma_{ac}^k \partial_l$

The right hand side of Equation  \ref{E:Ass} gives us: 
$\nabla_{\partial b}(\,\partial_a\, \circ\, \partial_c)= \nabla_{\partial b}(\,\Gamma_{ac}^k \partial_k)=\Gamma_{ac}^k\,\Gamma_{bk}^l\, \partial_l.$

The same method applies to $X=\partial_{\bar{a}}, Y=\partial_{\bar{b}}, Z=\partial_{\bar{c}}$. 

\smallskip 

So, we have a Frobenius algebra $(T_M^{1,0}\oplus T_M^{0,1},\circ)$. 

The rest of the construction follows the plan delimited above in proof of Theorem~ \ref{T:Main} .
\end{proof}

\subsection{On some properties}\label{S:Props}

\medskip 

\begin{prop}\label{P:Ricflat}
Let $M$ be a K\"ahler manifold with $c_1(M)=0$. If $M$ is covered by a complex torus
then $M$ is a pre-Frobenius manifold. If $M$ admits a flat submanifold $N$ then $N$ is a Frobenius manifold (possibly real).
\end{prop}
\begin{proof} Assume $M$ is a compact K\"ahler manifold with vanishing first Chern class $c_1(M)=0$. The first Chern class $c_1$ is given by  $\frac{\sqrt{-1}}{2\pi} R_{\alpha \bar{\beta}}dz^{\alpha} \wedge d z^{\bar{\beta}}$, 
where $R_{\alpha \bar{\beta}}$
Ricci tensor. Therefore, if $c_1=0$ then it admits a K\"ahler metric with vanishing Ricci tensor. 

\,

By Calabi--Yau and Fischer--Wolf \cite{Y2,Y4,FW}, if $M$ is an $n$-dimensional  compact K\"ahler manifold with $c_1=0$ then there exists a finite cover of $M$ which is the product $M_1\times M_2\times M_3$ of manifolds of respective dimensions $n_1,n_2, n_3\geq 0$ such that $\sum n_i=n$ where $M_1$ is a complex torus, $M_2$ is the product of some Calabi--Yau manifolds and $M_3$ the product of some hyperK\"ahler manifolds. 

\,

A compact K\"ahler manifold with $c_1=0$ admits a holomorphic projective connection if and only if it is covered by a complex torus. Assuming $M$ is covered by a complex torus this implies that $M$ admits a holomorphic affine structure. 

\, 

This statement follows from the following facts:
\begin{enumerate}
    \item  A compact K\"ahler--Einstein manifold admits a holomorphic projective connection if and only if it is of constant holomorphic sectional curvature \cite[Theorem 4.3]{Ko}.

\item A compact K\"ahler--Einstein manifold with vanishing curvature admits a holomorphic affine structure \cite[Cor. 4.5]{Ko}. 

\end{enumerate}


\,

We recover the other ingredients forming a pre-Frobenius bundle,
by mimicking the constructions presented in the previous sections/statements. 

\, 

\begin{itemize}
    \item The existence of a holomorphic affine structure implies the existence of a flat torsionless connection compatible  with a hermitian metric. This metric is K\"ahler in the flat coordinates. 

\item The construction of a multiplication operation on the tangent sheaf is obtained by using the compatible (non-degenerate) metric as well as the rank three symmetric tensor. This induces on the tangent sheaf the structure of a pre-Frobenius algebra. Therefore, we have a pre-Frobenius bundle. 
\end{itemize}
\smallskip 

The main difference between the pre-Frobenius and the Frobenius structure lies in the associativity of the algebra. Given a point $x\in M$, the algebra $(TM_x,\circ)$ is associative if and only if the curvature of $M$ at $x$ is null. Therefore, if $M$ admits a submanifold $N$ which is flat then $N$ is a Frobenius manifold.    
\end{proof}

\, 

From this statement we can deduce the following. 

\, 

\begin{cor}
   Let $M$ be a K\"ahler--Einstein compact manifold with $c_1(M$)=0. Suppose that it is covered by a complex torus. 
Then, $M$ admits a pre-Frobenius manifold structure. Suppose that the first Betti number $b_1(M)$ is non null. Then $M$ contains a submanifold $N$ being a Frobenius manifold. This Frobenius manifold coincides with the Riemannian flat torus.
\end{cor}
\begin{proof}
    This statement follows essentially from Theorem~\ref{P:Ricflat} and from \cite[Theorem 4.1 and Theorem 4.5]{FW}. The latter states that if $M$ is a compact connected Ricci-flat Riemannian manifold (i.e. satisfying the condition $c_1(M)=0$) then there exists a flat Riemannian torus $T$ of dimension greater or equal to the first Betti number of $M$ and a simply connected compact Ricci-flat Riemannian manifold $N$ as well as a finite Riemannian covering  such that $T\times N\to M$. 
\end{proof}
\,

\begin{lem}
A flat K\"ahler--Einstein submanifold $M$ of $\C^n$ is a totally geodesic Frobenius manifold.
\end{lem}

\begin{proof}
By assumption $M$ is flat. The vanishing of the curvature is a necessary condition to have holomorphic affine structure. So, $M$ is a holomorphic affine manifold.
Applying Proposition \ref{P:Kahler--Einstein}, we deduce that $M$ is a Frobenius manifold.
By Umehara~\cite{U}, any  K\"ahler--Einstein submanifold of $\C^n$ is totally geodesic. Therefore, if $M$ a flat K\"ahler--Einstein submanifold $M$ of $\C^n$ and if $M$ is flat then it is a flat totally geodesic manifold, satisfying the axioms of a Frobenius manifold. 
    
\end{proof}

\section{Classification of K\"ahler--Frobenius surfaces $(n=2)$}\label{S:2D}
We find the classes of surfaces in algebraic geometry carrying a Frobenius/ pre-Frobenius structure. 
 
\medskip 
\subsection{Conventions}
To each surface $S$ we associate numerical invariants:

\begin{itemize}
\item $K$ is the canonical divisor
    \item $q=h^{0,1}$ the irregulairty 
    \item $p_g(S)=h^{0,2}=h^{2,0}=P_1$: geometric genus 
    \item $P_n(S)=H^0(S,\cO_S(nK))$, the plurigenera of $S$. 
    \item Betti numbers $b_i=\dim H^{i}(S)$.
\end{itemize}

A minimal surface is to be understood in the sense that no exceptional curves of self intersection -1 are contained in it.

\subsubsection{Algebraic surfaces of general type}

\begin{lem}\label{L:AlgSurf}
If $M$ is a compact K\"ahler surface with universal covering biholomorphic to $\C^2$ then $M$ is a Frobenius manifold. 

\end{lem}
\begin{proof}
According to Kodaira, if $M$ is a compact K\"ahler surface with universal covering, biholomorphic to $\C^2$, then $M$ is a quotient $M=T/G$ of a complex torus $T$ by the free action of a finite group $G$ 
(which contains no translations). So, $M$ is a hyperelliptic surface.

\smallskip

By Kobayashi \cite{Ko}, $M$ has a holomorphic affine structure and by hypothesis $M$ is K\"ahler. So, all Chern classes vanish.

\smallskip

Now, we have all necessary ingredients to 
apply the construction outlined in the proof of the Theorem~\ref{T:Main}. This allows us to deduce that $M$ has a pre-Frobenius bundle structure. Given that manifolds of type $M=T/G$ are hyperelliptic surfaces and that those surfaces have vanishing curvature we can conclude that $M$ is flat. Therefore, $M$ is a Frobenius manifold.
\end{proof}

\medskip 

Lem.~\ref{L:AlgSurf} is a specific case of Iitaka's conjecture. The Iitaka conjecture states that if $M$ is a compact K\"ahler manifold having a universal covering biholomorphic to $\C^n$, then necessarily $M$ is a quotient $M=T/G$ of a complex torus $T$ by the free action of a finite group $G$ (which contains no translations). It is nowadays known that Iitaka's conjecture is true in low dimensions (dimension 2 and dimension 3).

\subsection{Classification of Frobenius surfaces}
\smallskip 

List of notation in the table below: 

-- (FM) stands for {\it Frobenius manifold}; 

-- (AS) stands for {\it affine holomorphic structure} ;

-- (K) stands for to {\it K\"ahler geometry}. 

We will prove the following. 

\medskip

\begin{Theorem}\label{T:ClassSurf}
    
~
\medskip 
There exist 8 compact K\"ahler Frobenius surfaces. 
The complex Frobenius surfaces are classified as in the table below: 

\,

\begin{center}

\begin{tabular}[ht]{|c|l|c|c|c|c|c|c|}
\hline 
 Kodaira &  Surface  type   & $b_i$ & $p_g$ & Chern  && &  \\
 class &  &  & & classes &  (FM)& (AS)& (K)\\
   \hline 

 &  &&&&& & \\

 & Complex tori & $b_1=4,\,b_2=6$ & $p_g=1$ & $c_1=c_2=0$ &  yes& yes & yes \\
&abelian surfaces&&&&&&  \\
 &  &&&&& & \\
 & Hyperelliptic surfaces & $b_1=2,\, b_2=2$ &0& $c_1=c_2=0$& yes&yes& yes \\

& &&&& && \\
$VIII_0$ & Minimal elliptic surfaces & odd $b_1$ & $p_g>0$ & $c_1^2=0,\, c_2=0$ & no & yes& no \\

& &&&&&&  \\

$VII_0$ & Minimal surfaces$^\star$  & $b_1=1$, $b_2=0$ & $p_g=0$ & & no & yes& no \\
& (Inoue surface) &&&& &&  \\
 & &&&& & &\\
$VII_0$ & Hopf surface  & $b_1=1$, $b_2=0$  &$p_g=0$ & $c_1=c_2=0$ & no  & yes& no\\

& covered by a primary   &&& &   && \\
& Hopf surface &&& &  
& & \\
&  such that $c(m-1)=0$  &&& & & & \\
& (see Sec.\ref{S:4.2}) &&&& && \\
&&&&& && \\
& Ruled & $b_1=2g,\, b_2=2$&0& $c_1^2=8(1-g)=$&no& no& yes\\
& surface && &$2c_2\neq 0$ & & & \\

& &&&&&  &\\

& K3 & $b_1=0,\, b_2=22$&1& $c_1=0$,\, $c_2=24$ & no & no& yes \\
& &&&& & &\\
 \hline
\end{tabular}  

 \end{center}

\end{Theorem}

\medskip 

\subsection{Proof of the statement}
Note that the classification on K\"ahler-Frobenius manifolds is related to the classification of Bieberbach groups and of crystallographic groups. We discuss each case below.

\subsubsection{Elliptic surfaces with even $b_1$}
\begin{prop}
Assume $\sE$ is an elliptic surface admitting an even first Betti number.  
\begin{enumerate}
    \item If $\sE$ is a complex torus, then it is a Forbenius manifold; 
\,

\,
  \item  if $\sE$ satisfies $h^{2,0}=0$ and $P_m\leq1$ for all $m>0$, then, $\sE$ carries the structure of a (pre-)Frobenius manifold. 
\end{enumerate}
 
\end{prop}
\begin{proof}
Assume $\sE$ is an elliptic surface with even first Betti number.

\smallskip 

By Miyaoka, an elliptic surface admits a K\"ahler metric iff its first Betti number is even. 

Therefore, by Schouten there exists  a Levi--Civita connection.
We have thus $(\sE,\Phi,\con{0})$ which provide the necessary ingredients for a pre-Frobenius manifold structure. 

It is a Frobenius manifold iff the curvature vanishes. This is the case for the complex torus. 
\end{proof}

\subsubsection{Elliptic surfaces of class VIII$_0$}

   An elliptic surface of class VIII$_0$ is a minimal elliptic surface with Betti number $b_1=1$, $p_g=0$ which is not a Hopf surface.

\medskip 

Although most of the elliptic surfaces with odd $b_1$ admit an affine structure this does not imply that they carry a Frobenius manifolds structure.
\begin{prop}
If $M$ is an elliptic surface of  class VIII$_0$
then $M$ does not admit the structure of a (pre-)Frobenius manifold. 
\end{prop}

\begin{proof}
If $M$ is an elliptic surface of class VIII$_0$
then it admits a holomorphic affine structure.
However, there is no existing K\"ahler metric. So, the necessary requirements for having a (pre-)Frobenius manifold are not fulfilled. 
 
\end{proof}

\subsubsection{Class  VII$_0$: Inoue surfaces}
 An elliptic surface of class VII$_0$ is a minimal elliptic surface with odd Betti number,\, $p_g>0$ and $c_1^2=0$.

Assume $M$ of class  VII$_0$ and that it is neither elliptic nor a Hopf surface.

Then, by Inoue's classification, there exists three main classes of such surfaces of the form $(\hH\times \C)/G$ where $\hH$ is the upper-half plane and $G$ is a group of affine transformations.  Since $G$ is a group of affine transformations hence Inoue's surface admit an affine structure. 

\medskip 

\begin{prop}
Let $M$ be an Inoue surface. Then it does not admit  a (pre-)Frobenius manifold structure. 
\end{prop}
\begin{proof}
  If $M$ is an Inoue surface (a surface of class VII$_0$ with $b_2=0$), then it admits holomorphic structure.  However, it does not admit a K\"ahler metric and it is not integrable. 
  
  \smallskip 
  
In fact, by Hasegawa, a four-dimensional solvmanifold admits a K\"ahler structure if and only if it is a complex torus or a hyperelliptic surface. 

\smallskip 

Therefore, $M$ cannot have a (pre-)Frobenius structure.   
\end{proof}

\medskip

\subsubsection{Class  VII$_0$: Hopf surfaces}\label{S:4.2}

Recall the following
\begin{dfn}
~
\medskip

\begin{itemize}
    \item[---] A Hopf surface is a compact complex surface whose universal covering space is biholomorphic to $\C^2-\{0\}$.  
    A Hopf surface is of the form $(\C^2-\{0\}/\langle g \rangle )$ where the group $g$ is generated by a polynomial contraction which can be written as:

 $$(x,y)\mapsto (ax+cy^m, by ) $$ 
where $a,b,c$ are complex numbers and $0<|a|\leq |b|<1$ and $(a-b^m)c=0$. 

   \item[---] A Hopf surface is primary if its fundamental group is infinite cyclic. 
\end{itemize} 
\end{dfn}

Every Hopf surface has a primary Hopf surface as a finite unramified covering space. 

\begin{prop}
 Hopf surfaces satisfying the condition that $c=0$ or $(m-1)=0$ admit a holomorphic affine structure and do not admit a (pre-)Frobenius manifold structure.   
\end{prop}
\begin{proof}
    The only case where the Hopf surface admits a holomorphic affine structure only if $c=0$ or $(m-1)=0$. However, the lack of K\"ahler metric is an obstruction to the existence of a Frobenius manifold. Therefore, this argument  disproves the existence of a Frobenius/ pre-Frobenius structure on Hopf surfaces. 

\end{proof}

\subsubsection{Ruled surfaces}
\begin{lem}
    Ruled surfaces  are not endowed with the structure of a pre-Frobenius manifold.
\end{lem}
\begin{proof}
    It is enough to see that ruled surfaces cannot carry any holomorphic affine structures. So, the conclusion is straightforward.  
\end{proof}
\subsubsection{K3 surfaces}
K3 surfaces do not admit a holomorphic affine structure so they cannot be Frobenius manifolds nor pre-Frobenius manifolds.

\subsubsection{Case of $p_g=0$ and $q\geq 1$: Hyperelliptic surfaces}

A hyperelliptic surface is a compact complex surface, not isomorphic to an abelian surface, which admits a finite \'etale covering by an abelian surface. These surfaces were classified by Enriques–Severi and Bagnera-de Franchis.
The case of hyperelliptic surfaces is discussed in Lem~\ref{L:AlgSurf}. In particular, flat K\"ahler surfaces {\it include} the class of hyperelliptic surfaces.  

\smallskip

In the 2-dimensional case, the group $G$ is cyclic.  The classification of those hyperelliptic surfaces is well known and we recall it. Let $E$ and $F$ be elliptic curves (1-dimensional complex torus). Let $G$ be a group of translations of $E$ acting on $F$. The hyperelliptc surfaces are the following type $(E\times F)/G$ where:
 \begin{enumerate}
     \item$ G=\Z_2$ acts on $F$ by symmetry 
     \item $G=\Z_2\oplus \Z_2$ acts on $F$ by $x\mapsto -x$ , $x\mapsto x+ \epsilon$,  where $\epsilon$ is a point of order 2. 
     \item $G=\Z_4$ acting on 
     $F=\C/(\Z\oplus \Z\imath)$ by $x\mapsto \imath x$
     \item $G=\Z_4\oplus \Z_2$  acting on  $F=\C/(\Z\oplus \Z\imath)$ by 
     $x\mapsto \imath x$, $x\mapsto x+\frac{1+\imath}{2}$
     \item $G=\Z_3$ where $F= \C/(\Z\oplus \Z\rho)$ where $\rho^3=1$ and $\rho\neq1$ and $G$ acts by $x\mapsto \rho x$
     \item $G=\Z_3 \oplus \Z_3$ acts by 
     $x\mapsto \rho x$ and 
     $x\mapsto x+\frac{1-\rho}{3}$
     \item $G=\Z_6$ acts by $x\mapsto -\rho x$.
     
 \end{enumerate}

\begin{lem}
  There exist eight K\"ahler-Frobenius surfaces. These surfaces are complex tori and hyperelliptic surfaces.    
\end{lem}
\begin{proof}
One can classify all flat K\"ahler surfaces, according to their holonomy group. The holonomy groups are 1, $\Z_2, \Z_3, \Z_4, \Z_6$. There exist eight flat Ka\"hler surfaces with such holonomy groups. A flat K\"ahler surface is a K\"ahler-Frobenius surface. Therefore, the statement is proved.
   
\end{proof}

\section{Classification of K\"ahler--Frobenius manifolds $n\geq 2$}\label{S:ClassN}
\subsection{}
A classification theorem of compact K\"ahler Frobenius manifolds is provided. Below the table we recall some definitions of Calabi--Yau manifolds and generalized Hantzsche--Wendt manifolds.

\begin{Theorem}\label{T:Class}
Let $n>1$. The set of $n$-dimensional K\"ahler-Frobenius manifolds are classified as follows: 
  
  \centering 
    \begin{itemize} 
\item[]\begin{table}[ht]
    \centering
    \begin{tabular}{|c l|}
    \hline
        &\\
        &\\
          
       $T=\C^n/\Z^n$, $n>1$  & Complex torus  \\
          & \\
          
       $\C^n/\Lambda$,   & $\Lambda$ is a torsion-free complex crystallographic group.\\
       &  \\

       $T/G$, & $G$ a finite group acting on the compact complex torus $T$, \\
       &freely and which contains no translations.\\
          & \\
          
     Hyperelliptic manifolds  &flat K\"ahler manifolds \\
     &  not isomorphic to $T$.\\
    & \\
  General. Hantzsche--Wendt manifolds$^\star$ &  manifolds equipped with a rotation group,\\
   & isomorphic to $(\Z/2\Z)^{n-1}$ and $b_1=0$.\\
   & The manifolds carry a spin structure.\\
      & \\
      
3D Calabi--Yau flat K\"ahler manifolds &   non-trivial holonomy group.  \\
   & \\
   
Calabi--Yau manifolds, & $b_1=\cdots=b_{2n-1}=0$ and $b_n=2^n$  
\\
odd dimension, $n>2$ &  
$b_{2k}$=$n\choose k$ for $0\leq k\leq n.$\\
&\\
    \hline
    \end{tabular}
    \caption{Sources of Frobenius K\"ahler-manifolds}
    \label{tab:my_label}

\end{table}
\end{itemize}
\end{Theorem}

\subsection{}
Recall the definition of a generalized Hantzsche--Wendt manifolds.


A Hantzsche--Wendt $n$-manifold is a (non-necessarily) orientable flat  manifold in dimension $n\geq 3$, with holonomy group  $\Z_2^{n-1}$. This generalises the original version of a  Hantzsche--Wendt manifold, which is a 3-dimensional manifold with vanishing first Betti number $b_1=0$. The non-orientable case occurs for each dimension $n\geq2$. The orientable case occurs only in each odd dimension $n\geq 3.$ Any complex orientable Hantzsche--Wendt manifold has a spin structure. The holonomy representation $\Z^{n-1}_2\to U(n)$ of an orientable complex Hantzsche--Wendt manifold has its image in $SU(n)$.

\, 

\subsection{Proof of Theorem \ref{T:Class}}
\begin{proof}
To prove this theorem it is enough to verify that the manifolds listed in the table are flat K\"ahler manifolds. This can be easily checked by hand and also has been proved in the constructions of \cite{DHS,R}. Therefore, the manifolds listed in the table are flat K\"ahler manifolds.
  It remains now to apply Theorem~\ref{T:Main}. The statement shows that a compact complex K\"ahler manifold with vanishing curvature is Frobenius. Therefore, we have a list of classes of K\"ahler-Frobenius manifolds. 
This ends the proof. 
\end{proof}

\begin{cor}
In low dimensions there exists finitely many topological types of K\"ahler-Frobenius manifolds. In particular one can list all those classes of manifolds.  

  \begin{itemize}
       \item there exist 8 compact K\"ahler 2-dimensional Frobenius manifolds.  \\
       \item There exist 174  compact 3-dimensional K\"ahler Frobenius manifolds.  
  \end{itemize}
\end{cor}
\begin{proof}
By Theorem \ref{T:Class} one has a classification of different sources of K\"ahler--Frobenius manifolds. Using the computations on the number classes of flat K\"ahler manifolds of dimensions 2 and 3 and the classification criterion in \cite{DHS} we deduce that there exist 8 compact 2-dimensional K\"ahler--Frobenius manifolds and  174  compact 3-dimensional K\"ahler--Frobenius manifolds.  
\end{proof}
\subsection{Hermitian--Einstein Space of Connections}
We discuss the construction depicted throughout these sections.

We prove that:

\begin{Theorem}\label{T:HE}
  The pencil of connections on the tangent bundle to a K\"ahler--Frobenius manifold with trivial canonical bundle and $c_1=0$ forms a pencil of Hermitian--Einstein connections.   
  
\end{Theorem}

\,

\begin{proof}
    
Assume $E$ is a holomorphic vector bundle of rank $r$ over a complex manifold $M$ of dimension $n$. Assume $(z_1,\cdots, z_n)$ are local coordinates on $M$ and that $M$ is endowed with a metric given by $ds^2=2\sum g_{j\bar{k}}dz^jd\bar{z}^k$.

\,

Let $h$ be a Hermitian fibre--metric on $\pazocal{E}$. Assume $\{s_1,\cdots, s_r\}$ are linearly independent local holomorphic sections of $\pazocal{E}$. Put $h_{i,\bar{j}}=h(s_i,{\bar{s}_j})$.

\,

With respect to the sections $\{s_1,\cdots, s_r\}$ and the local coordinates $(z_1,\cdots, z_n)$ in $M$ the curvature of the Hermitian connection in the bundle $\pazocal{E}$ is given by

\,

\[\frac{\partial^2 h_{\alpha\bar{\beta}}}{\partial z^j \partial \bar{z}^k}-h^{e\bar{\gamma}} \frac{\partial h_{\alpha\bar{\gamma} }}{\partial z^j}\frac{\partial h_{e\bar{\beta}}}{\partial \bar{z}^k}.\]

\,

The quadruple $(\sE,M,h,ds^2)$ is a Hermitian--Einstein vector bundle if the Ricci tensor is proportional to $h$: 
$$R_{i\bar{j}}=c\cdot h_{i\bar{j}},$$ where $c$ is a constant.

\, 

By Kobayashi, if the quadruple $(E,M,h,ds^2)$ is an Hermitian--Einstein vector bundle then it is {\it not} instable (in the sense of Bogomolov \cite{Ko2}). Now, by Uhlenbeck--Yau \cite{UY}, a {\it stable} vector bundle over a compact K\"ahler manifold admits a unique Hermitian--Einstein metric.

\,

By the Kobayashi--Hitchin correspondence \cite{Ko3,Hit1}, stable vector bundles over a complex manifold are related to Einstein--Hermitian vector bundles. In particular,  going back on more time to Kobayashi's statement, if $M$ is a compact K\"ahler--Einstein manifold  the tangent bundle $TM$ is not instable and therefore, we can conclude by the above argument that $TM$ is a  Einstein--Hermitian vector bundle.

\,
 
On holomorphic vector bundles over compact complex manifolds, we have the existence of Hermitian--Yang--Mills connections (or Hermitian--Einstein metrics). The Hermitian--Yang--Mills connection is a specification of the Yang--Mills connection to the case of {\it hermitian} vector bundles over a complex manifold.  
 
\,

Let $\nabla_0$ be a Hermitian--Einstein connection. Let $\cA$ be the affine space of connections. Any connection $\nabla\in \cA$ can be written as $\nabla = \nabla_0 + t \eta$, where $\eta$ is a section in
 $\Gamma(T^*M\otimes \pazocal{E}^*\otimes \pazocal{E})$ and $t$ is a real parameter. The curvature $F_\nabla=d\nabla-\nabla\wedge \nabla$ of the Hermitian connection is a section of $End \pazocal{E} \otimes \Lambda^2 T^*M$. 
 
 \, 
 
 A Hermitian--Yang--Mills connection is a unitary connection $\nabla$ for the Hermitian metric $h$ which satisfies 
$F_\nabla^{0,2}=0$ i.e. vanishing of the curvature and the trace $tr$ of the curvature satisfies the following relation $tr(F_\nabla)=\kappa Id_\pazocal{E}.$

\,

 Given a K\"ahler metric on $M$, define the operation $tr: \Gamma (End \pazocal{E} \otimes T_{1,1}^* M) \to \Gamma(End \pazocal{E})$, where 
$ tr(F)=\sum g^{j\bar{k}}F_{aj\bar{k}}^b$. 

\, 

We have a Hermitian--Yang--Mills if there exists a Hermitian metric $h$ for which the Hermitian curvature $F$ satisfies $tr(F)= \kappa Id$.
 If we take $\pazocal{E} = TM$ to be the holomorphic tangent bundle and $h = g$ then the condition above becomes the familiar K\"ahler-Einstein condition (definition from Uhlenbeck--Yau).

The curvature of the deformed connections is given by the following formula $\tilde{F}=F_{\nabla_0+\lambda\eta}= F_{\nabla_0}+ d_{\nabla_0} + 1/2[\eta\wedge \eta].$ This is a Hermitian--Einstein connection if $Tr(F_{\nabla+\lambda\eta})=0.$
 
By definition $tr(F_{\nabla_0+\lambda\eta})=\sum g^{j\bar{k}}\tilde{F}_{aj\bar{k}}^b$.

By hypothesis, $\tilde{F}_{aj\bar{k}}^b$ is flat because it is a Frobenius manifold. 
So, $tr(\tilde{F})=0.$

Therefore, we can conclude that pencil of connections  on the Frobenius manifold is a pencil of Hermitian Yang--Mills connections.

\end{proof}
\subsection{Conclusion and remarks}
To conclude, we have determined which classes of K\"ahler manifolds satisfy the associativity equations, implied by the Frobenius manifold structure. Certain classes of such K\"ahler--Frobenius manifolds have direct relations to theta functions, which can be interesting both from an analysis perspective and a number theory perspective.

We want to remark that our construction connects also to Joyce structures \cite{B,J07}, which form a close notion to Frobenius manifold. A Joyce structure on $M$ involves a pencil of flat, meromorphic Ehresmann connections on the tangent bundle $TM$. Somehow, the flatness condition involves there a meromorphic function (called a Joyce function).

\end{document}